\documentclass[12pt,twoside]{article}
\usepackage{amsmath,amsfonts,amssymb,a4,enumitem,epsfig,array,psfrag}
\usepackage{url}
\usepackage{amsthm}
\usepackage{subfig}
\usepackage{chngcntr} 
\usepackage{color}
\usepackage[section]{placeins}

\usepackage[T1]{fontenc}          
\usepackage{float}

\DeclareMathAlphabet{\mathpzc}{OT1}{pzc}{m}{it}

\newcommand{\R}{\mathbb{R}}
\newcommand{\Z}{\mathbb{Z}}

\newtheorem{theo}{Theorem}[section]

\newtheorem{Conj}[theo]{Conjecture}
\newtheorem{proposition}[theo]{Proposition}

\newtheorem{example}[theo]{Example}

\newtheorem{lemma}[theo]{Lemma}

\def\SS{\mathbb S}

\def\1{\mathbf 1}
\def\k{\mathbf k}
\def\j{\mathbf j}
\def\i{\mathbf i}

\begin{document}
\title{Characterization of some convex curves\\
 on the $3$-sphere}
\author{
Em\'ilia Alves }

\maketitle

\begin{abstract}
In this paper we provide a characterization for a class of convex curves on the $3$-sphere. 
More precisely, using a theorem that represents a locally convex curve on the $3$-sphere as a pair of curves in $\SS^2$, one
of which is locally convex and the other is an immersion, we are capable of completely characterize a class of convex curves on the 3-sphere.





\end{abstract}

\section{Introduction}
A curve $\gamma: [0,1] \rightarrow \SS^n$ of class $C^k$ ($k \geq n$) is called \emph{locally convex} if 
\[ \mathrm{det}(\gamma(t),\gamma'(t),\gamma''(t), \cdots, \gamma^{(n)}(t)) > 0 \] 
for all $t$. 
Therefore, a curve of class $C^k$, for $k \geq 2$, on the $2$-sphere is locally convex if its geodesic curvature is positive at every point. 
Analogously, a curve of class $C^k$, for $k \geq 3$, on the $3$-sphere is locally convex if its geodesic torsion is always positive (see proposition~\ref{lccs3} for a proof).

Given a locally convex curve $\gamma:[0,1] \rightarrow \SS^n$, we associate a \emph{Frenet frame curve} $ \mathcal{F_{\gamma}}:[0,1] \rightarrow \mathrm{SO}_{n+1}$ by applying the Gram-Schmidt orthonormalization to the $(n+1)$-vectors $(\gamma(t),\gamma'(t),\dots,\gamma^{(n)}(t))$.
Given $Q \in \mathrm{SO}_{n+1}$, we denote by ${\mathcal{L}\SS^{n}}(Q)$ the set of all locally convex curves $\gamma:[0,1] \rightarrow \SS^n$ such that $\mathcal{F}_\gamma(0)=I$ and $\mathcal{F}_\gamma(1)=Q$, where $I$ denotes the identity matrix.

The study of the spaces of locally convex curves on the $2$-sphere started with J. A. Little in $1970$. In~\cite{Lit70} he proved that the space ${\mathcal{L}\SS^{2}}(I)$ has $3$ connected components: ${\mathcal{L}\SS^{2}}(\mathbf{1}),  {\mathcal{L}\SS^{2}}(-\mathbf{1})_c$ and ${\mathcal{L}\SS^{2}}(-\mathbf{1})_n$ (this notation will be clarified below). Here we denoted by ${\mathcal{L}\SS^2}(-\mathbf{1})_n$ the component associated with non-convex curves whereas ${\mathcal{L}\SS^2}(-\mathbf{1})_c$ denotes the component of convex curves~\cite{Fen29}, see Figure~\ref{Little}. Notice that this component is contractible~\cite{Ani98}.

\begin{figure}[H]
\centering
\includegraphics[scale=0.3]{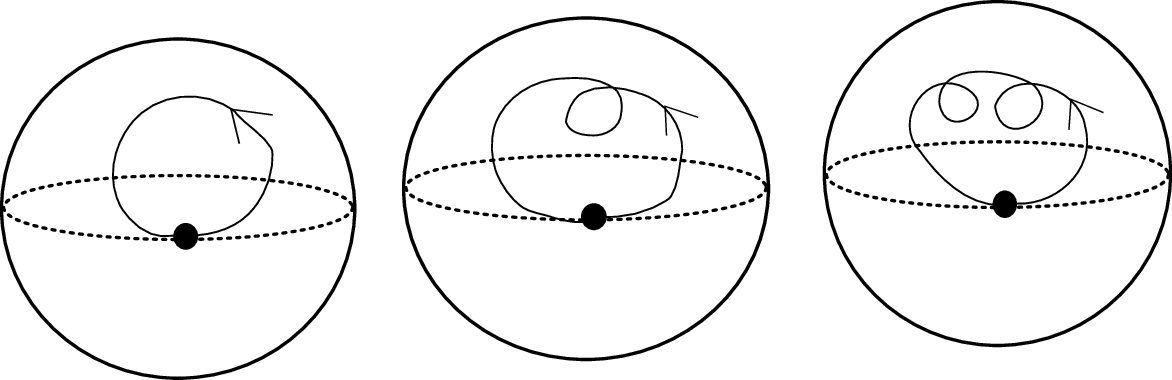}
\caption{Examples of curves in the components $\mathcal{L}\SS^2(-\mathbf{1})_c, \; \mathcal{L}\SS^2(\mathbf{1})$ and $\mathcal{L}\SS^2(\mathbf{-1})_n$, respectively.}
\label{Little}
\end{figure}


For $n \geq 2$, the universal (double) cover of $\mathrm{SO}_{n+1}$ is the spin group $\mathrm{Spin}_{n+1}$; let 
$\Pi_{n+1} : \mathrm{Spin}_{n+1} \rightarrow \mathrm{SO}_{n+1} $ be the natural projection. 
Let us denote by $\mathbf{1}$ the identity element in $\mathrm{Spin}_{n+1}$, and by $-\mathbf{1}$ the unique non-trivial element in $\mathrm{Spin}_{n+1}$ such that $\Pi_{n+1}(-\mathbf{1})=I$. 
Therefore, the Frenet frame curve $\mathcal{F_{\gamma}}:[0,1] \rightarrow \mathrm{SO}_{n+1}$ can be uniquely lifted to a continuous curve $\tilde{\mathcal{F}_{\gamma}}:[0,1] \rightarrow \mathrm{Spin}_{n+1}$ such that $\mathcal{F_{\gamma}}=\Pi_{n+1} \circ \tilde{\mathcal{F}_{\gamma}}$ and $\tilde{\mathcal{F}_{\gamma}}(0)=\mathbf{1}$. Given $z \in \mathrm{Spin}_{n+1}$, we denote by
${\mathcal{L}\SS^{n}}(z)$ the set of curves
$\gamma \in {\mathcal{L}\SS^{n}}(\Pi_{n+1}(z))$
for which $\tilde{\mathcal{F}_{\gamma}}(1)=z$.
Obviously, ${\mathcal{L}\SS^{n}}(\Pi_{n+1}(z)) = {\mathcal{L}\SS^{n}}(z)\sqcup {\mathcal{L}\SS^{n}}(-z).$
Notice that $\mathcal{L}\SS^n(z)$ tuns out to be non-empty.

The motivation to study these spaces of curves comes from the realm of ordinary differential equations (ODE), since the space of locally convex curves on the $n$-sphere is deeply related to the study of linear ODEs of order $n+1$; see~\cite{Alv16},~\cite{BST03},~\cite{BST06},~\cite{BST009} and~\cite{SalTom05}. As we already mention, J. Little was the precursor of the study of the homotopy type of the spaces of locally convex curves on the $2$-sphere. After him, the study of these spaces in higher dimensional spheres and also related spaces (for example, in the Euclidean space and in the Projective space) regain interested in the nineties; here we mention the work of B. Z. Shapiro, M. Z. Shapiro and B. A. Khesin that determined the number of connected components of those spaces; see~\cite{SS91},~\cite{Sha93},~\cite{KS92} and~\cite{KS99}.
At this time, although the number of connected components of those spaces has 
been completely understood, little information on the cohomology or higher homotopy groups was available, 
even for $n=2$. 

The homotopy type of the spaces of locally convex curves on the $2$-sphere was completely determined with the work~\cite{Sal13} of N. Saldanha, which followed important developments reported in~\cite{Sal09I} and \cite{Sal09II}. 
Today this topic continues to attract the attention of several authors working on a variety of problems not just because of its topological richness, but also due to the number of spillovers it has across several disciplines of Mathematics and applications. These include, but are not limited to, symplectic geometry~\cite{Arn95}, control theory~\cite{RS90} and engineering~\cite{Dub57}.
In spite of the attention the topic has received, the homotopy type of the spaces $\mathcal{L}\SS^n(z)$, $n \geq 3$ and $z \in \mathrm{Spin}_{n+1}$, remains a open problem. 

Recently in~\cite{AlvSal19} the authors present some partial results about the homotopy type of the spaces of locally convex curves on the $3$-sphere. In this paper we resort to some results obtained in~\cite{AlvSal19} as to better understand these spaces of curves and refine the analysis of the homotopy type of $\mathcal{L}\SS^3(z)$, $z \in \mathrm{Spin}_{4}$. For this, we  revisit some concepts and then state our contributions.


It is well known that $\mathrm{Spin}_{4}$ can be identified with $\SS^3 \times \SS^3$.
So, given $z=(z_l,z_r) \in \SS^3 \times \SS^3$
(where $l$ and $r$ stand for \emph{left} and \emph{right})
we denote by $\mathcal{L}\SS^3(z_l,z_r)$ the space of locally convex curves in $\SS^3$ with the initial and final lifted Frenet frame respectively $(\mathbf{1},\mathbf{1})$ and $(z_l,z_r),$ i.e.,
\[ \mathcal{L}\SS^3(z_l,z_r)=\{\gamma: [0,1] \rightarrow \SS^3 \; | \, \tilde{\mathcal{F}}_\gamma(0)=(\mathbf{1},\mathbf{1}) \; \text{and} \; \tilde{\mathcal{F}}_\gamma(1)=(z_l,z_r) \}. \] 

In~\cite{AlvSal19} it was proved that every locally convex curve $\gamma \in \mathcal{L}\SS^3(z_l,z_r)$ can be represented as a pair of curves on the $2$-sphere $\gamma_l$ and $\gamma_r$, where $\gamma_l$ is locally convex and $\gamma_r$ is merely an immersion. A locally convex curve in $\SS^3$ is rather hard to describe from a geometrical point of view; and the above result allows us to understand such a curve as a pair of curves in $\SS^2$, a situation where geometrical intuition is easier. 

If $\gamma \in \mathcal{L}\SS^3(z_l,z_r)$ is such that its left part $\gamma_l \in \mathcal{L}\SS^2(z_l)$ is convex, then $\gamma$ is convex (Proposition 5.7 in~\cite{AlvSal19}), but in general the converse is false. 

Yet there are still some spaces in which one can characterize completely the convexity of $\gamma$ by merely considering its left part.    

The first space in which we have such a characterization is $\mathcal{L}\SS^3(-\1,\k)$.

\begin{theo}\label{th2}
A curve $\gamma \in \mathcal{L}\SS^3(-\1,\k)$ is convex if and only if its left part $\gamma_l \in \mathcal{L}\SS^2(-\1)$ is convex.
\end{theo}

The second space is $\mathcal{L}\SS^3(\1,-\1)$; however, for this space, we can only give a necessary condition for a curve to be convex, even though we believe that this condition is also sufficient. 

\begin{theo}\label{th3}
Assume that $\gamma \in \mathcal{L}\SS^3(\1,-\1)$ is convex. Then its left part $\gamma_l \in \mathcal{L}\SS^2(\1)$ is contained in an open hemisphere and its rotation number is equal to $2$.
\end{theo}

The remainder of this work is structured as follows.
Section~\ref{preliminaries} is devoted to some algebraic preliminaries, fundamental in this work. 
In Section~\ref{section3} we present some basic definitions and properties of locally convex curves. 
In Subsection~\ref{s33} we define a large space of curves, the space of generic curves, and properly define the Frenet frame curve associated with a locally convex curve and to a generic curve.
Finally in Subsection~\ref{convexcurves} we define globally convex curves, which are of fundamental importance in the study
of locally convex curves and in this work.
In Section~\ref{examples} we will give some examples of locally convex curves on the $3$-sphere that will be fundamental in the proofs of Theorem~\ref{th2} and Theorem~\ref{th3}: for this, we will review some definitions and results that are contained in~\cite{AlvSal19}.
In Section~\ref{sectionproofs} we finally give the proofs of Theorem~\ref{th2} and Theorem~\ref{th3} that characterizes convexity in the spaces $\mathcal{L}\SS^3(-\1,\k)$ and $\mathcal{L}\SS^3(\1,-\1)$, respectively. 
In Subsection~\ref{s64} there is a precise definition of a curve contained in an open hemisphere and its rotation number.

\noindent{\bf Acknowledgements:} The author is grateful to Nicolau C. Saldanha for helpful conversations and 
to CAPES, CNPq, FAPERJ and PUC-Rio for the financial support. 

\section{Algebraic Preliminaries} \label{preliminaries}

In this section we briefly review some algebraic concepts: first we recall the definition and some properties about the spin groups. Then we review the decomposition of the orthogonal groups $\mathrm{SO}_{n+1}$ and the spin groups $\mathrm{Spin}_{n+1}$ in the (signed) Bruhat cells.

\subsection{Spin Groups}

For $n \geq 2$, the universal (double) cover of $\mathrm{SO}_{n+1}$ is the spin group $\mathrm{Spin}_{n+1}$; let 
$\Pi_{n+1} : \mathrm{Spin}_{n+1} \rightarrow \mathrm{SO}_{n+1} $ be the natural projection. 
The group $\mathrm{Spin}_{n+1}$ is therefore a simply connected Lie group, which is also compact, and it has the same Lie algebra (and hence the same dimension) as $\mathrm{SO}_{n+1}$. 
Let us denote by $\1$ the identity element in $\mathrm{Spin}_{n+1}$, and by $-\mathbf{1}$ the unique non-trivial element in $\mathrm{Spin}_{n+1}$ such that $\Pi_{n+1}(-\1)=I$.


We will give a brief description of $\mathrm{Spin}_{n+1}$ in the cases $n=2$ and $n=3$ since it will be fundamental in this work. It is well known that $\mathrm{Spin}_3 \simeq \SS^3$ and $\mathrm{Spin}_4 \simeq \SS^3 \times \SS^3$.  


First, let us recall the definition of the algebra of quaternions:
\[ \mathbb{H}:=\{a\1+b\i+c\j+d\k \; | \; (a,b,c,d) \in \R^4\}, \]
where $\1 \in \mathbb{H}$ is the multiplicative unit, and $\i,\j,\k$ satisfy the product rules
$ \i^2=\j^2=\k^2=\i\j\k=-1$. 
As a real vector space, $ \mathbb{H}$ is isomorphic to $\R^4$, hence one can define a Euclidean norm on $\mathbb{H}$ and the set of quaternions with unit norm, $U(\mathbb{H})$, can be naturally identified with $\SS^3$. The space of imaginary quaternions (i.e., of real part $0$) is naturally identified with $\R^3$.

The natural projection $\Pi_3 : \mathrm{Spin}_3 \rightarrow \mathrm{SO}_3$ is given by $\Pi_3(z)h = zh\bar{z}$, for any $h \in \R^3$.
The natural projection $\Pi_4 : \mathrm{Spin}_4 \rightarrow \mathrm{SO}_4$ is given by $\Pi_3(z_l,z_r)q = z_lq\bar{z_r}$, for any $q \in \R^4$. For a matrix-notation approach of these objects we refer the reader to~\cite[Subsection $2.1$]{AlvSal19}.

\subsection{Bruhat cells and the Coxeter-Weyl Group} \label{BruhatDecomp}

We denote by $\mathrm{Up}^+_{n+1}$ the group of upper triangular matrices with positive diagonal entries and by $\mathrm{Up}^1_{n+1} \subset \mathrm{Up}^+_{n+1}$ the subgroup of upper triangular matrices with diagonal entries equal to one.


The map $B : \mathrm{Up}^1_{n+1} \times \mathrm{SO}_{n+1} \mapsto \mathrm{SO}_{n+1}$ defined by
$ B(U,Q)=UQU' $
where $U'$ is the unique matrix in $\mathrm{Up}^+_{n+1}$ such that $UQU' \in \mathrm{SO}_{n+1}$ is called the \emph{Bruhat action}. 
The Bruhat action is clearly a group action of $\mathrm{Up}^1_{n+1}$ on $\mathrm{SO}_{n+1}$, 
and we call its finitely many orbits the  \emph{(signed) Bruhat cells}. 
It turns out that, two matrices $Q \in \mathrm{SO}_{n+1}$ and $Q' \in \mathrm{SO}_{n+1}$ belong to the same Bruhat cell if and only if there exist $U$ and $U'$ in $\mathrm{Up}^+_{n+1}$ such that $Q'=UQU'$.
In other words, given $Q \in \mathrm{SO}_{n+1}$ the Bruhat cell of $Q$ is the set of matrices $UQU' \in \mathrm{SO}_{n+1}$, where $U$ and $U'$ belong to $\mathrm{Up}^+_{n+1}$. We denote by $\mathrm{Bru}_Q$ the \emph{(signed) Bruhat cell} of $Q \in \mathrm{SO}_{n+1}$.





Let $\mathrm{B}_{n+1} \subset \mathrm{O}_{n+1}$ the Coxeter-Weyl group of signed permutations matrices, 
so $|\mathrm{B}_{n+1}| = 2^{n+1}(n+1)!$.
Let $\mathrm{B}_{n+1}^+ = \mathrm{B}_{n+1} \cap \mathrm{SO}_{n+1}$ the group of signed permutations matrices of determinant one, so $|\mathrm{B}_{n+1}^+| = 2^{n}(n+1)!$.
Given a signed permutation matrix in $\mathrm{B}_{n+1}^+$ we can drop its entries signs and define a homomorphism from $\mathrm{B}_{n+1}^+$ to the symmetric group $S_{n+1}$.
Each Bruhat cell contains a unique signed permutation matrix $P \in \mathrm{B}_{n+1}^+$; it follows that, given any two Bruhat cells, associated with two distinct signed-permutation matrices, they are disjoint. 
Therefore we have the Bruhat decomposition of $\mathrm{SO}_{n+1}$:
\[ \mathrm{SO}_{n+1}=\bigsqcup_{P \in \mathrm{B}_{n+1}^+}\mathrm{Bru}_P, \]
and there are $2^n(n+1)!$ different Bruhat cells. 

Since the group $\mathrm{Up}_{n+1}^1$ is contractible, its Bruhat action on $\mathrm{SO}_{n+1}$ lifts to a Bruhat action on $\mathrm{Spin}_{n+1}$ that, for simplicity, we still denote by $B : \mathrm{Up}_{n+1}^1 \times \mathrm{Spin}_{n+1} \rightarrow \mathrm{Spin}_{n+1}$. 
As before, the Bruhat cells on $\mathrm{Spin}_{n+1}$ are the orbits of the Bruhat action.

Let us denote by $\tilde{\mathrm{B}}_{n+1}^+ = \Pi_{n+1}^{-1}(\mathrm{B}_{n+1}^+) \subset \mathrm{Spin}_{n+1}$, so $|\tilde{\mathrm{B}}_{n+1}^+| = 2^{n+1}(n+1)!$. Let $z \in \mathrm{Spin_{n+1}}$; we define the \emph{Bruhat cell} $\mathrm{Bru}_z$ as the connected component of $\Pi_{n+1}^{-1}(\mathrm{Bru}_{\Pi_{n+1}(z)})$ which contains $z$. 
Obviously $\Pi_{n+1}^{-1}(\mathrm{Bru}_{\Pi_{n+1}(z)}) = \mathrm{Bru}_z \sqcup \mathrm{Bru}_{-z}$, where each set $\mathrm{Bru}_{z}$, $\mathrm{Bru}_{-z}$ is contractible and non empty.

Since the group $\mathrm{Up}_{n+1}^1$ is contractible, its Bruhat action on $\mathrm{SO}_{n+1}$ lifts to a Bruhat action on $\mathrm{Spin}_{n+1}$ that, for simplicity, we still denote by $B : \mathrm{Up}_{n+1}^1 \times \mathrm{Spin}_{n+1} \rightarrow \mathrm{Spin}_{n+1}$. As before, the Bruhat cells on $\mathrm{Spin}_{n+1}$ are the orbits of the Bruhat action.

Therefore the Bruhat decomposition of $\mathrm{SO}_{n+1}$ can be lifted to the universal cover $\Pi_{n+1}: \mathrm{Spin}_{n+1} \rightarrow \mathrm{SO}_{n+1}$ and we have the Bruhat decomposition for $\mathrm{Spin}_{n+1}$:
\[ \mathrm{Spin}_{n+1}=\bigsqcup_{\tilde{P} \in \tilde{B}_{n+1}^+}\mathrm{Bru}_{\tilde{P}}, \]
and there are $2^{n+1}(n+1)!$ disjoint Bruhat cells in $\mathrm{Spin}_{n+1}$.

\section{Basic Definitions and Properties} \label{section3}

In this section we gather basic notions and properties on the spaces of curves under analysis.
In what follows, we define a large space of curves (the space of generic curves), then we characterize locally convex curves on the $3$-sphere and finally we define globally convex curves.

\subsection{Frenet frame curves and Generic curves}\label{s33}


Consider a locally convex curve $\gamma:[0,1] \rightarrow \SS^n$.
Applying Gram-Schmidt orthonormalization to the $(n+1)$-vectors $(\gamma(t),\gamma'(t),\dots,\gamma^{(n)}(t))$, 
there exists a unique $\mathcal{F}_\gamma(t) \in \mathrm{SO}_{n+1}$ and $R_\gamma(t) \in \mathrm{Up}_{n+1}^+$ such that
\begin{equation}\label{frenet_gram}
(\gamma(t),\gamma'(t),\dots,\gamma^{(n)}(t))=\mathcal{F_{\gamma}}(t) R_\gamma(t),
\end{equation}
where we recall that $\mathrm{Up}_{n+1}^+$ is the space of upper triangular matrices with positive diagonal entries and real coefficients. 
The curve $\mathcal{F_{\gamma}}:[0,1] \rightarrow \mathrm{SO}_{n+1}$ defined by~\eqref{frenet_gram} is called the \emph{Frenet frame curve} of the locally convex curve $\gamma:[0,1] \rightarrow \SS^n$. 

Notice that the definition of Frenet frame curve associated with a locally convex curve $\gamma$ on the $n$-sphere does not depend on the choice of a positive (or orientation preserving) reparametrization: indeed a computation using the chain rule shows that
\[ (\gamma \circ \phi(t),(\gamma \circ \phi)'(t),\dots,(\gamma \circ \phi)^{(n)}(t))=(\gamma(\tau),\gamma'(\tau),\dots,\gamma^{(n)}(\tau))U, \]
where $\tau=\phi(t)$ is a positive reparametrization (i.e., the sign of $\phi'(t)$ is positive), and where $U$ is an upper triangular matrix whose diagonal is given by $(1,\phi'(t), \dots,(\phi'(t))^n)$. So $U \in \mathrm{Up}_{n+1}^+$, which implies that $\mathcal{F_{\gamma \circ \phi}}(t)=\mathcal{F_{\gamma}}(\tau)$. 

We denote by ${\mathcal{L}\SS^{n}}$ the set of all locally convex curves $\gamma : [0,1] \rightarrow \SS^{n}$ such that $\mathcal{F}_\gamma(0)=I$. Obviously, ${\mathcal{L}\SS^{n}}(Q) \subset {\mathcal{L}\SS^{n}}$. 



Many authors already discussed the topological structures of the spaces $\mathcal{L}\SS^n(Q)$.
It is well known that different topological structures give different spaces which are homotopically equivalent.
Therefore, we will consider that our curves are smooth. We notice, however, that even if juxtaposition of curves jeopardizes smoothness, there is no loss of generality in assuming so.
For more about this discussion we refer to the reader~\cite{SS12},~\cite{Sal13} or~\cite{Alv16}.

\medskip

Even though we will be mainly interested in locally convex curves, it will
be useful in the sequel to consider a larger space of curves.

\medskip

A curve $\gamma: [0,1] \rightarrow \SS^n$ of class $C^k$ ($k \geq n$) is called \emph{generic} if 
the vectors $\gamma(t), \gamma'(t),\gamma''(t),\dots,\gamma^{(n-1)}(t)$ are linearly independent for all $t \in [0,1]$.
Given $\gamma$ a generic curve on the $n$-sphere we can still define its Frenet frame curve. 
In fact, by applying Gram-Schmidt orthonormalization to the linearly independent $n$-vectors $\gamma(t),\gamma'(t),\dots,\gamma^{(n-1)}(t)$ we obtain $n$ orthonormal vectors $u_0(t),u_1(t), \dots, u_{n-1}(t)$ and then, there is a unique vector $u_n(t)$ for which $u_0(t), u_1(t), \dots, u_{n-1}(t), u_n(t)$ is a positive orthonormal basis. 
We may thus set
\begin{equation}\label{frenet_gram2}
\mathcal{F_{\gamma}}(t)=(u_0(t),u_1(t), \dots, u_{n-1}(t), u_n(t)) \in \mathrm{SO}_{n+1}
\end{equation}
and make the following more general definition. The curve $\mathcal{F_{\gamma}}:[0,1] \rightarrow \mathrm{SO}_{n+1}$ defined by~\eqref{frenet_gram2} is called the \emph{Frenet frame curve} of the generic curve $\gamma:[0,1] \rightarrow \SS^n$. 
Clearly, the latter definition coincides with the former when $\gamma$ is locally convex.

Given $Q \in \mathrm{SO}_{n+1}$, we denote by $\mathcal{G}\SS^n(Q)$ the set of all generic curves $\gamma:[0,1] \rightarrow \SS^{n}$ such that $\mathcal{F_{\gamma}}(0)=I$ and $\mathcal{F_{\gamma}}(1)=Q$. For $z \in \mathrm{Spin}_{n+1}$, we define ${\mathcal{G}\SS^{n}}(z)$ as the subset of ${\mathcal{G}\SS^{n}}(\Pi_{n+1}(z))$ for which $\tilde{\mathcal{F}_{\gamma}}(1)=z$.
Obviously, $\mathcal{L}\SS^n(Q) \subset \mathcal{G}\SS^n(Q)$ and $\mathcal{L}\SS^n(z) \subset \mathcal{G}\SS^n(z)$.

The homotopy type of the spaces $\mathcal{G}\SS^{n}(z)$, $z \in \mathrm{Spin}_{n+1}$, is well-known. Let us define $\Omega \mathrm{Spin}_{n+1}(z)$ to be the space of all continuous curves $\alpha : [0,1] \rightarrow \mathrm{Spin}_{n+1}$ with $\alpha(0)=\mathbf{1}$ and $\alpha(1)=z$. It is well understood that different values of $z \in \mathrm{Spin}_{n+1}$ does not change the space $\Omega \mathrm{Spin}_{n+1}(z)$ up to homeomorphism, therefore we will usually drop $z$ from the notation and write $\Omega \mathrm{Spin}_{n+1}$ instead of $\Omega \mathrm{Spin}_{n+1}(z)$. Follows from the works of Hirsch and Smale (\cite{Hir59} and \cite{Sma59b}) that the \emph{Frenet frame injection} $\tilde{\mathcal{F}}: \mathcal{G}\SS^{n}(z) \rightarrow \Omega \mathrm{Spin}_{n+1}$ defined by $(\tilde{\mathcal{F}}(\gamma))(t) = \tilde{\mathcal{F}}_\gamma(t)$ is a homotopy equivalence (see Subsection $5.2$ of~\cite{AlvSal19} for more on this).

\medskip

Let us look at the special case where $\gamma$ is a generic curve on the $2$-sphere, i.e., $\gamma$ is an immersion.
Let us denote by $\mathbf{t}_\gamma(t)$ the unit tangent vector of $\gamma$ at the point $\gamma(t)$, that is 
$\mathbf{t}_\gamma(t):=\frac{\gamma'(t)}{||\gamma'(t)||} \in \SS^2,$
and by $\mathbf{n}_\gamma(t)$ be the unit normal vector of $\gamma$ at the point $\gamma(t)$, that is 
$ \mathbf{n}_\gamma(t):= \gamma(t) \times \mathbf{t}_\gamma(t) $
where $\times$ is the cross-product in $\R^3$. We then have
$ \mathcal{F_{\gamma}}(t)=(\gamma(t),\mathbf{t}_\gamma(t),\mathbf{n}_\gamma(t)) \in \mathrm{SO}_3 $
where $\mathbf{t}_\gamma(t)$ is the unit tangent and $\mathbf{n}_\gamma(t)$ the unit normal we defined above.  
The geodesic curvature $\kappa_\gamma(t)$ is by definition
$ \kappa_\gamma(t):=\mathbf{t}_\gamma'(t)\cdot\mathbf{n}_\gamma(t) $
where $\cdot$ is the Euclidean inner product. Here's a geometric definition of locally convex curves in $\SS^2$.

\begin{proposition}
A generic curve $\gamma:[0,1] \rightarrow \SS^2$ is locally convex if and only if $\kappa_\gamma(t)> 0$ for all $t \in (0,1)$. 
\end{proposition}

\begin{proof}
For a proof see Proposition 18 in~\cite{Alv16}.
\end{proof}


Let us now consider a generic curve $\gamma$ on the $3$-sphere and let $e_1,e_2,e_3,e_4$ be the canonical basis of $\R^4$. 
It is clear that
$ \mathcal{F}_\gamma(t)e_1=\gamma(t), \quad \mathcal{F}_\gamma(t)e_2=\mathbf{t}_\gamma(t)=\frac{\gamma'(t)}{||\gamma'(t)||}.  $
We define the unit normal $\mathbf{n}_\gamma(t)$ and binormal $\mathbf{b}_\gamma(t)$ by the formulas
$ \mathbf{n}_\gamma(t)=\mathcal{F}_\gamma(t)e_3, \, \mathbf{b}_\gamma(t)=\mathcal{F}_\gamma(t)e_4  $
so that  
\[ \mathcal{F_{\gamma}}(t)=(\gamma(t),\mathbf{t}_\gamma(t),\mathbf{n}_\gamma(t),\mathbf{b}_\gamma(t)) \in \mathrm{SO}_4. \]
The geodesic curvature $\kappa_\gamma(t)$ is still defined by
$ \kappa_\gamma(t):=\mathbf{t}_\gamma'(t)\cdot\mathbf{n}_\gamma(t) $
but we further define the geodesic torsion $\tau_{\gamma}(t)$ by
$ \tau_\gamma(t):=-\mathbf{b}_\gamma'(t)\cdot\mathbf{n}_\gamma(t). $
It is clear that the geodesic curvature is never zero. We can then characterize locally convex curves in $\SS^3$.

\begin{proposition} \label{lccs3}
A generic curve $\gamma:[0,1] \rightarrow \SS^3$ is locally convex if and only if $\tau_\gamma(t)> 0$ for all $t \in (0,1)$. 
\end{proposition}

\begin{proof}
Without loss of generality, we may assume that $\gamma$ is parametrized by arc-length. Then, as before, $\mathbf{t}_\gamma(t)=\gamma'(t)$ and $\kappa_\gamma(t)=\gamma''(t)\cdot\mathbf{n}_\gamma(t)$. Also
$ \gamma''(t)=-\gamma(t)+\kappa_\gamma(t)\mathbf{n}_\gamma(t) $
and hence
$ \gamma'''(t)=-\gamma'(t)+\kappa_\gamma(t)'\mathbf{n}_\gamma(t)+\kappa_\gamma(t)\mathbf{n}_\gamma'(t). $
Since $\mathbf{b}_\gamma(t)\cdot \mathbf{n}_\gamma(t)=0$, we have
$ \tau_\gamma(t)=-\mathbf{b}_\gamma'(t)\cdot \mathbf{n}_\gamma(t)=\mathbf{b}_\gamma(t)\cdot \mathbf{n}'_\gamma(t). $
One then easily computes
\[ \gamma'''(t)\cdot\gamma(t)=0, \]
\[ \gamma'''(t)\cdot\gamma'(t)=-1-\kappa_\gamma(t)^2, \]
\[ \gamma'''(t)\cdot\mathbf{n}_\gamma(t)=\kappa_\gamma'(t), \]
\[ \gamma'''(t)\cdot\mathbf{b}_\gamma(t)=\kappa_\gamma(t)\tau_\gamma(t). \]
So we have the equality
$ (\gamma(t),\gamma'(t),\gamma''(t),\gamma'''(t))=\mathcal{F}_\gamma(t)R_\gamma(t) $
with
\[ 
R_\gamma(t)=
\begin{pmatrix}
1 & 0 & -1 & 0 \\
0 & 1 & 0 & -1-\kappa_\gamma(t)^2\\
0 & 0 & \kappa_\gamma(t) & \kappa_\gamma'(t) \\
0 & 0 & 0 & \kappa_\gamma(t)\tau_\gamma(t).
\end{pmatrix}. \]
Since $\mathcal{F}_\gamma(t)$ has determinant $1$ and $\kappa_\gamma(t)$ is never zero, we have
\[ \mathrm{det}(\gamma(t),\gamma'(t),\gamma''(t),\gamma'''(t))=\mathrm{det}R_\gamma(t)=\kappa_\gamma(t)^2\tau_\gamma(t) \]
and this proves the statement.
\end{proof}


We continue collecting preliminary notions and auxiliary results instrumental to our arguments.
Let $ \Gamma : [0,1] \rightarrow \mathrm{SO}_{n+1} $, the logarithmic derivative of the curve $\Gamma$ is defined as
$ \Lambda(t)=(\Gamma(t))^{-1}\Gamma'(t) $, that is,  $\Gamma'(t)=\Gamma(t)\Lambda(t). $
Notice that  $\Lambda$ belongs to the Lie algebra of $\mathrm{SO}_{n+1}$. 
Therefore $\Lambda(t)$ is automatically a skew-symmetric matrix for all $t \in [0,1]$.

Let $\mathfrak{J} \subset \mathfrak{so}_{n+1}$ be the set of Jacobi matrices, i.e., $\mathfrak{J}$ is the set of tridiagonal skew-symmetric matrices with positive subdiagonal entries, that is, matrices of the form
\[ \begin{pmatrix}
0 & -c_1 & 0 & \ldots & 0 \\
c_1 & 0 & -c_2 &  & 0 \\
 & \ddots & \ddots & \ddots &  \\
0 &  & c_{n-1} & 0 & -c_n  \\
0 &  &  0 & c_n & 0
\end{pmatrix}, \quad c_1>0, \dots, c_n>0. \]

We are interested in the following definition in order to characterize the Frenet frame curves associated with a locally convex curves. Let us call a curve $\Gamma : [0,1] \rightarrow \mathrm{SO}_{n+1}$ \emph{Jacobian} if its logarithmic derivative $\Lambda(t)=(\Gamma(t))^{-1}\Gamma'(t)$ belongs to $\mathfrak{J}$ for all $t \in [0,1]$. 

In~\cite{SS12} was introduced a homeomorphism between locally convex curves in $\mathcal{L}\SS^n$ and Jacobian curves starting at the identity, more precisely:

\begin{proposition}\label{propjacobian}
Let $\Gamma : [0,1] \rightarrow \mathrm{SO}_{n+1}$ be a smooth curve with $\Gamma(0)=I$. Then $\Gamma$ is Jacobian if and only if there exists $\gamma \in \mathcal{L}\SS^n$ such that $\mathcal{F}_\gamma=\Gamma$.
\end{proposition}

The Lemma 2.1 in~\cite{SS12} gives us a correspondence as follows: 
given $\Gamma$ a Jacobian curve with $\Gamma(0)=I$, then if we define $\gamma_\Gamma$ by setting $\gamma_\Gamma(t)=\Gamma(t)e_1$ then $\gamma_\Gamma \in \mathcal{L}\SS^n$ and conversely, given $\gamma \in \mathcal{L}\SS^n$, its Frenet frame curve is a Jacobian curve.
Notice that the Frenet frame curve $\mathcal{F}_\gamma$ uniquely determines the curve $\gamma$.


For instance, let us characterize Frenet Frame curves of locally convex curves on the $3$-sphere that will be useful in this work. 
If $\gamma : [0,1] \rightarrow \SS^3$ is locally convex, then 
$ \mathcal{F_{\gamma}}(t)=(\gamma(t),\mathbf{t}_\gamma(t),\mathbf{n}_\gamma(t),\mathbf{b}_\gamma(t)) \in \mathrm{SO}_4 $
and one gets
\begin{equation}\label{logderives3}
\Lambda_\gamma(t)=
\begin{pmatrix}
0 & -||\gamma'(t)|| & 0  & 0 \\
||\gamma'(t)|| & 0 & -||\gamma'(t)||\kappa_\gamma(t) & 0  \\
0 & ||\gamma'(t)||\kappa_\gamma(t) & 0 & -||\gamma'(t)||\tau_\gamma(t) \\
0 & 0 & ||\gamma'(t)||\tau_\gamma(t) & 0
\end{pmatrix}.
\end{equation}

\subsection{Convex curves} \label{convexcurves}

Next we introduce a special class of locally convex curves of fundamental importance in the study of the space of locally convex curves.

Consider $\gamma : [0,1] \rightarrow \SS^n$, a smooth curve, and let $H \subseteq \mathbb{R}^{n+1}$ be a hyperplane. 
Let us call $\gamma$ a \emph{globally convex} curve if the image of $\gamma$ intersects $H$, counting with multiplicity, in at most $n$ points. 
This definition requires us to clarify the notion of multiplicity. First, endpoints of the curve are not counted as intersections. So, if there exists $t \in (0,1)$ such that $\gamma(t) \in H$, then the multiplicity of the intersection point $\gamma(t)$ is the smallest integer $k \geq 1$ such that
$ \gamma^{(j)}(t) \in H, \quad 0 \leq j \leq k-1. $ 
Therefore, a multiplicity is one if $\gamma(t) \in H$ but $\gamma'(t) \notin H$, it is two if $\gamma(t) \in H$, $\gamma'(t) \in H$ but $\gamma''(t) \notin H$, and so on.

From the definition it is easy to prove that every globally convex curve is locally convex. 
For ease of presentation, we refer to globally convex curves as convex curves.



Given any $n \geq 2$ and $z \in \mathrm{Spin}_{n+1}$ recall that $\mathcal{L}\SS^n(z)$ is the space of locally convex curves in $\SS^n$ whose initial lifted Frenet frame is $\1$ and the final lifted Frenet Frame is $z$.




Convexity is strongly related with the number of connected components of the spaces $\mathcal{L}\SS^n(z)$, $z \in \mathrm{Spin}_{n+1}$. Let us recall also the following result.

\begin{theo}[M. Z. Shapiro, \cite{Sha93}, S. Anisov, \cite{Ani98}]
\label{thmani}
The space $\mathcal{L}\SS^n(z)$ has exactly two connected components if there exist convex curves in $\mathcal{L}\SS^n(z)$, and one otherwise. If $\mathcal{L}\SS^n(z)$ has two connected components, one is made of convex curves, and this component is contractible.   
\end{theo}

This result highlights the importance of identifying the existence of convex curves in the spaces of locally convex curves. 
It will be critical in what follows.

\section{Examples} \label{examples}


In the recent paper~\cite{AlvSal19} the authors produced a homemorphism between the space of $\gamma \in \mathcal{L}\SS^3(z_l,z_r)$ and the space of pairs of curves $(\gamma_l,\gamma_r) \in \mathcal{L}\SS^2(z_l) \times \mathcal{G}\SS^2(z_r)$ for which such conditions are verified, more precisely:




\begin{theo}[Alves and Saldanha, 2019] \label{th1}
There exists a homeomorphism between the space $ \mathcal{L}\SS^3(z_l,z_r)$ and the space of pairs of curves  $(\gamma_l,\gamma_r) \in \mathcal{L}\SS^2(z_l)\times \mathcal{G}\SS^2(z_r)$
satisfying the condition
\begin{equation*}\label{cond}\tag{L}
||\gamma_l'(t)||=||\gamma_r'(t)||, \quad \kappa_{\gamma_l}(t)>|\kappa_{\gamma_r}(t)|, \quad t \in [0,1].
\end{equation*}
\end{theo}

For the proof we refer to \cite[Subsection 4.1]{AlvSal19}. 
The Theorem~\ref{th1} allows us to represents a locally convex curve in $\SS^3$ as a pair of a locally convex curve in $\SS^2$ and an immersion in $\SS^2$, with some compatibility conditions. Hence to produce examples of locally convex curve in $\SS^3$, it is enough to produce examples of such pairs. 
In this section, we want to use this theorem to produce examples in the spaces we are interested in: 
namely, $\mathcal{L}\SS^3(-\1,\k)$ and $\mathcal{L}\SS^3(\1,-\1)$. 
These examples will be fundamental in the proof of Theorem~\ref{th2} and Theorem~\ref{th3}.

In the sequel, for the sake of completeness, we recall the notations and main steps developed 
in the Subsection 4.2 of~\cite{AlvSal19}. 

Let $c$ be a real number such that $0<c \leq 2\pi$. 
Consider $\sigma_c : [0,1] \rightarrow \SS^2$ the unique circle of length $c$, that is 
$||\sigma_c'(t)||=c$, with initial and final Frenet frame equals to the identity. 
Let $\rho \in (0,\pi/2]$ be the radius of curvature and $c=2\pi\sin\rho$; the curve $\sigma_c$ can be given by the following formula
\[ \sigma_c(t)=\cos\rho(\cos\rho,0,\sin\rho)+\sin\rho(\sin\rho\cos(2\pi t), \sin(2\pi t), -\cos\rho\cos(2\pi t)). \] 
The geodesic curvature of the curve $\sigma_c$ is given by $\cot(\rho) \in [0,+\infty)$. 
Note that, for $0<c<2\pi$, $\sigma_c$ is locally convex but also convex, but for $c=2\pi$, this is a meridian curve
\[ \sigma_{2\pi}(t)=(\cos(2\pi t), \sin(2\pi t), 0), \]
which has zero geodesic curvature, so this is just an immersion.
 
All our examples will be constructed as follows. For the left part of our curves, we will use $\sigma_c$ with $c<2\pi$ and iterate it a certain number of times, and for the right part of curves, we will use $\sigma_{2\pi}$ and iterate it a certain number of times. Since the right part will always have zero geodesic curvature, the only restriction so that this pair of curves defines a locally convex curve in $\SS^3$ is the condition that their length should be equal. However, in order to realize different final lifted Frenet frame, we will have to iterate the curve $\sigma_c$ (on the left) and the curve $\sigma_{2\pi}$ (on the right) a different number of times: the equality of length will be achieved by properly choosing $c$ in each case. 
Define the curve $\sigma_c^m$, for $m>0$, as the curve $\sigma_c$ iterated $m$ times, that is $\sigma_c^m(t)=\sigma_c(mt)$, $t \in [0,1]$.
See the Figure~\ref{fig:b} below for an illustration.  
\begin{figure}[H]
\centering
\includegraphics[scale=0.3]{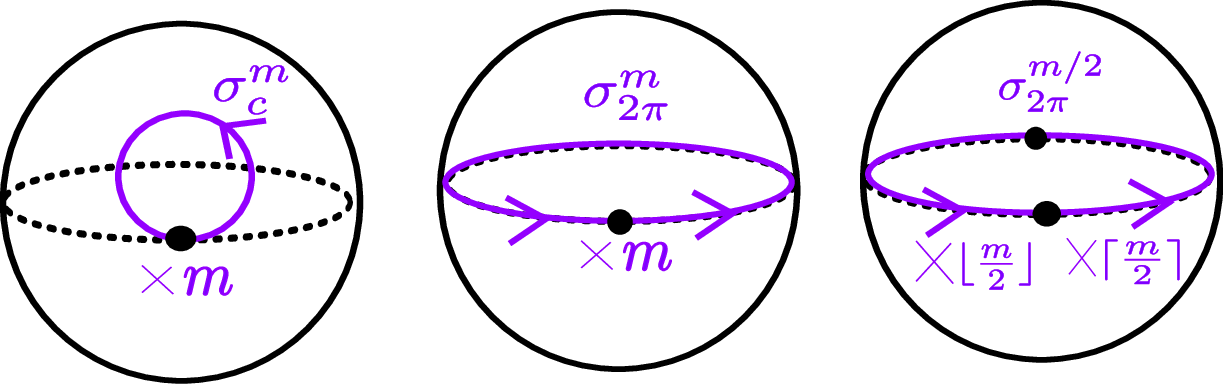}
\caption{The curves $\sigma_c^m$, $\sigma_{2\pi}^m$ and $\sigma_{2\pi}^{m/2}$.}
\label{fig:b}
\end{figure}

\begin{example} \label{family1}
Let us give explicit examples in the spaces $\mathcal{L}\SS^3((-\1)^m,\k^m)$, $m \geq 1$. For $m\equiv 1$ or $2$ modulo $4$, this will give examples in the spaces $\mathcal{L}\SS^3(-\1,\k)$ and $\mathcal{L}\SS^3(\1,-\1)$.
\end{example}
For $m\equiv 1$ or $2$ modulo $4$, we want to define a curve $\gamma_1^m \in \mathcal{L}\SS^3((-\1)^m,\k^m)$ 
such that its left and right parts are given by
\[ \gamma_{1,l}^m=\sigma_c^m \in \mathcal{L}\SS^2((-\1)^m), \quad \gamma_{1,r}^m =\sigma_{2\pi}^{m/2} \in \mathcal{G}\SS^2(\k^m). \]
To define a pair of curves, we need to choose $0<c<2\pi$ such that
\[ ||(\gamma_{1,l}^{m})'(t)||=||(\sigma_c^{m})'(t)||=cm \]
is equal to 
\[ ||(\gamma_{1,r}^{m})'(t)||=||(\sigma_{2\pi}^{m/2})'(t)||=\pi m. \]
It suffices to choose $c=\pi$ so that both curves have length equal to $\pi m$, then the geodesic curvature of $\gamma_{1,l}^m=\sigma_c^m$ is constantly equal to $\sqrt{3}$. Clearly, the geodesic curvature of $\gamma_{1,r}^m =\sigma_{2\pi}^{m/2}$ is zero.  

Let us now find explicitly the curve $\gamma_1^m$. From Theorem~\ref{th1}, we can compute
\[ ||(\gamma_1^{m})'(t)||=\frac{||(\gamma_{1,l}^m)'(t)||(\kappa_{\gamma_{1,l}}(t)-\kappa_{\gamma_{1,r}}(t))}{2}=\frac{m\pi\sqrt{3}}{2} \]
\[ \kappa_{\gamma_1^m}(t)=\frac{2}{\kappa_{\gamma_{1,l}}(t)-\kappa_{\gamma_{1,r}}(t)}=\frac{2}{\sqrt{3}},\] 
\[\tau_{\gamma_1^m}(t)=\frac{\kappa_{\gamma_{1,l}}(t)+\kappa_{\gamma_{1,r}}(t)}{\kappa_{\gamma_{1,l}}(t)-\kappa_{\gamma_{1,r}}(t)}=1. \]
Therefore the logarithmic derivative of $\gamma_1^m$ is constant and given by
\[  
\Lambda_{\gamma_1^m}=\frac{\pi}{2}
\begin{pmatrix}
0 & -m\sqrt{3} & 0  & 0 \\
m\sqrt{3} & 0 & -2m & 0  \\
0 & 2m & 0 & -m\sqrt{3} \\
0 & 0 & m\sqrt{3} & 0
\end{pmatrix}.
\]
The Jacobian curve $\Gamma_{\gamma_1^m}$ satisfies
\[ \Gamma_{\gamma_1^m}'(t)=\Gamma_{\gamma_1^m}(t)\Lambda_{\gamma_1^m}, \quad \Gamma_{\gamma_1^m}(0)=I \]
and can also be computed explicitly since it is the exponential of $\Gamma_{\gamma_1^m}$, that is
\[ \Gamma_{\gamma_1^m}(t)=\exp(t\Lambda_{\gamma_1^m}). \] 
The curve $\gamma_1^m$ is then equal to $\Gamma_{\gamma_1^m}e_1$, and we find that

\begin{eqnarray*}
\gamma_1^m(t) & = & \left(\frac{1}{4}\cos\left(\frac{3}{2}t\pi m\right)+\frac{3}{4}\cos\left(\frac{1}{2}t\pi m\right) \right.,\\
&  & \frac{\sqrt{3}}{4}\sin\left(\frac{3}{2}t\pi m\right)+\frac{\sqrt{3}}{4}\sin\left(\frac{1}{2}t\pi m\right), \\ 
&  & \frac{\sqrt{3}}{4}\cos\left(\frac{1}{2}t\pi m\right)-\frac{\sqrt{3}}{4}\cos\left(\frac{3}{2}t\pi m\right), \\
&  & \left.\frac{3}{4}\sin\left(\frac{1}{2}t\pi m\right)-\frac{1}{4}\sin\left(\frac{3}{2}t\pi m\right)\right).
\end{eqnarray*}

Below we give an illustration in the case $m=5$ (Figure~\ref{fig:g}).

%
%
%

\begin{figure}[H]
\centering
\includegraphics[scale=0.3]{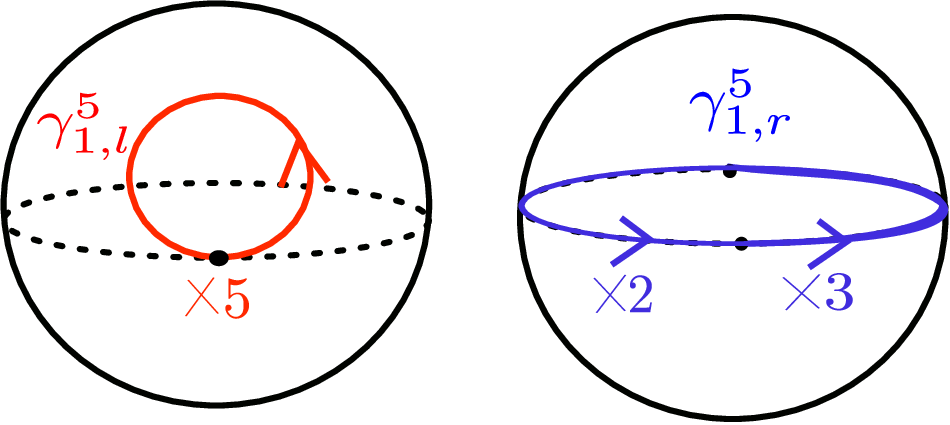}
\caption{The curve $\gamma_1^5 \in \mathcal{L}\SS^3(-\1,\k)$, where $\gamma_{1,l}^5=\sigma_{\pi}^5$ and $\gamma_{1,r}^5=\sigma_{2\pi}^{5/2}$ .}
\label{fig:g}
\end{figure}

\begin{proposition} \label{curve gamma -1,k is convex} 
The curve $\gamma_1^1 \in \mathcal{L}\SS^3(-\1,\k)$ is convex.
\end{proposition}
\begin{proof}
We will prove that the curve $\gamma_{1}^1 \in \mathcal{L}\SS^3(-\1,\k)$ defined in Example~\ref{family1} is convex (see Figure~\ref{fig:c}). 

\begin{figure}[H]
\centering
\includegraphics[scale=0.3]{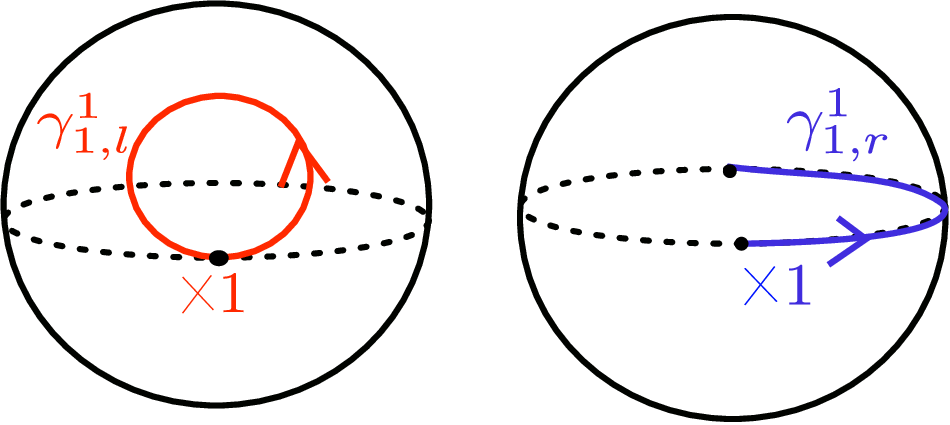}
\caption{The curve $\gamma_1^1 \in \mathcal{L}\SS^3(-\1,\k)$, where $\gamma_{1,l}^1=\sigma_{\pi}^1$ and $\gamma_{1,r}^1=\sigma_{2\pi}^{1/2}$.}
\label{fig:c}
\end{figure}

Up to a reparametrization with constant speed, this curve is the same as the curve $\tilde{\gamma} : [0,\pi/2] \rightarrow \SS^3$ defined by
\begin{eqnarray*}
\tilde{\gamma}(t) & = & \left(\frac{1}{4}\cos\left(3t\right)+\frac{3}{4}\cos\left(t\right)\right., \\
&  & \frac{\sqrt{3}}{4}\sin\left(3t \right)+\frac{\sqrt{3}}{4}\sin\left(t\right), \\
&  & \frac{\sqrt{3}}{4}\cos\left(t\right)-\frac{\sqrt{3}}{4}\cos\left(3t\right), \\
&  & \left.\frac{3}{4}\sin\left(t\right)-\frac{1}{4}\sin\left(3t \right)\right).
\end{eqnarray*} 
Note that, it is sufficient to prove that $\tilde{\gamma}$ is convex. 
Observe that for $t \in [0,\pi/2)$, the first component of $\tilde{\gamma}$ never vanishes, so if we define the central projection
\[ p : (x_1,x_2,x_3,x_4) \in \R^4 \mapsto \left(1,\frac{x_2}{x_1},\frac{x_3}{x_1},\frac{x_4}{x_1}\right), \]
then it is sufficient to prove that the curve $p(\tilde{\gamma})$, defined for $t \in [0,\pi/2)$ is convex.
We compute
\[ p(\tilde{\gamma}(t))=\left(1,\sqrt{3}\tan t,\sqrt{3}(\tan t)^2,(\tan t)^3\right) \]
and hence, if we reparametrize by setting $x=\tan t$, we obtain the curve
\[ x \in [0,+\infty) \mapsto (1,\sqrt{3}x,\sqrt{3}x^2,x^3) \in \R^4. \]
It is now obvious that this curve is convex, and therefore our initial curve $\gamma_1^1$ is convex.
\end{proof}

\begin{proposition} \label{curve gamma 1,-1 is convex} 
The curve $\gamma_1^2 \in \mathcal{L}\SS^3(\1,-\1)$ is convex.
\end{proposition}

\begin{proof}
The proof of this assertion is entirely similar to the proof of the fact that $\gamma_1^1 \in \mathcal{L}\SS^3(-\1,\k)$ (see Figure~\ref{fig:d}) is convex, therefore we let this proof for the reader.

\begin{figure}[H]
\centering
\includegraphics[scale=0.3]{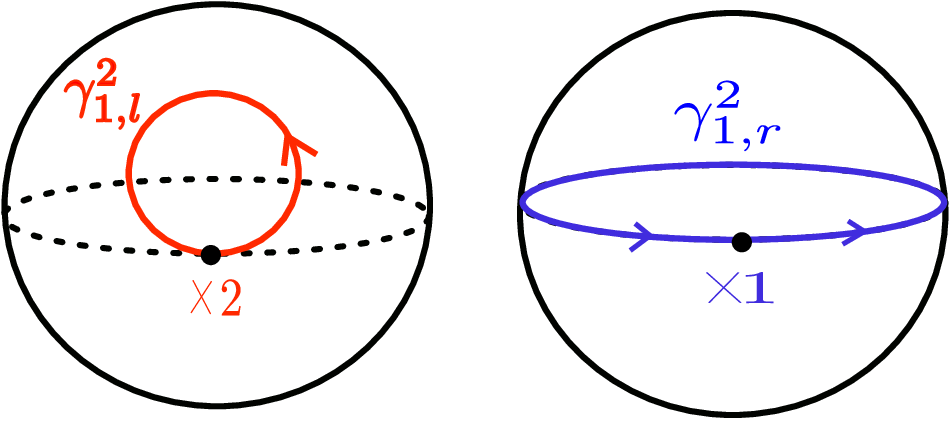}
\caption{The curve $\gamma_1^2 \in \mathcal{L}\SS^3(\1,-\1)$, where $\gamma_{1,l}^2 = \sigma_\pi^2$ and $\gamma_{1,r}^2 = \sigma_{2\pi}^{1}$.}
\label{fig:d}
\end{figure}
\end{proof}

Note that the curve $\gamma_1^2$ is a example of curve in the space $\mathcal{L}\SS^3(\1,-\1)$ that is convex but its left part $\sigma_\pi^2$ is not convex.

\section{Characterization of convex curves in the spaces $\mathcal{L}\SS^3(-\1,\k)$ and $\mathcal{L}\SS^3(\1,-\1)$}\label{sectionproofs}

Our nearest goal is to prove Theorem~\ref{th2}, i.e., a curve $\gamma \in \mathcal{L}\SS^3(-\1,\k)$ is convex if and only if $\gamma_l \in \mathcal{L}\SS^2(-\1)$ is convex. In the sequel, we will prove Theorem~\ref{th3}, i.e., if $\gamma \in \mathcal{L}\SS^3(\1,-\1)$ is convex then its left part $\gamma_l \in \mathcal{L}\SS^2(\1)$ is contained in an open hemisphere and its rotation number is 2. 




\subsection{Proof of Theorem~\ref{th2} }\label{s63}

In this subsection, we give the proof of Theorem~\ref{th2} which characterizes convexity in the spaces $\mathcal{L}\SS^3(-\1,\k)$ by merely considering the left part of the curve.

\begin{proof}[Proof of Theorem~\ref{th2}]
Recall that we want to prove that a curve $\gamma \in \mathcal{L}\SS^3(-\1,\k)$ is convex if and only its left part $\gamma_l \in \mathcal{L}\SS^2(-\1)$ is convex.

It is clear that $\mathcal{L}\SS^2(-\1)$ contains convex curves; the curve $\sigma_c$, for $0 < c < 2\pi$ defined in Section~\ref{examples} is convex, since it intersects any hyperplane of $\R^3$ (or equivalently any great circle) in exactly two points. Using Theorem~\ref{thmani}, the space $\mathcal{L}\SS^2(-\mathbf{1})$ has therefore $2$ connected components,
\[ \mathcal{L}\SS^2(-\mathbf{1})=\mathcal{L}\SS^2(-\mathbf{1})_c \sqcup \mathcal{L}\SS^2(-\mathbf{1})_n, \]
where $\mathcal{L}\SS^2(-\mathbf{1})_c$ is the component associated with convex curves and $\mathcal{L}\SS^2(-\mathbf{1})_n$ the component associated with non-convex curves.

The space $\mathcal{L}\SS^3(-\1,\k)$ also contains convex curves. Indeed, we prove in the Proposition~\ref{curve gamma -1,k is convex} that the curve $\gamma_{1}^1 \in \mathcal{L}\SS^3(-\1,\k)$ is convex.
Therefore, again using the Theorem~\ref{thmani}, the space $\mathcal{L}\SS^3(-\mathbf{1},\k)$ has therefore $2$ connected components,
\[ \mathcal{L}\SS^3(-\mathbf{1},\k)=\mathcal{L}\SS^3(-\mathbf{1},\k)_c \sqcup \mathcal{L}\SS^3(-\mathbf{1},\k)_n, \]
where $\mathcal{L}\SS^3(-\mathbf{1},\k)_c$ is the component associated with convex curves and $\mathcal{L}\SS^3(-\mathbf{1},\k)_n$ the component associated with non-convex curves.

Then we can use Theorem~\ref{th1} to define a continuous map
\[ L : \mathcal{L}\SS^3(-\mathbf{1},\k) \rightarrow \mathcal{L}\SS^2(-\mathbf{1})  \]
by setting $L(\gamma)=\gamma_l$, where $(\gamma_l,\gamma_r)$ is the pair of curves associated with $\gamma$. Since $L$ is continuous and $\mathcal{L}\SS^3(-\mathbf{1},\k)_c$ is connected, its image by $L$ is also connected. Moreover, we know $\gamma_1^1 \in \mathcal{L}\SS^3(-\mathbf{1},\k)_c$, and that $L(\gamma_1^1)=\sigma_1 \in \mathcal{L}\SS^2(-\mathbf{1})_c$, therefore the image of $\mathcal{L}\SS^3(-\mathbf{1},\k)_c$ by $L$ intersects $\mathcal{L}\SS^2(-\mathbf{1})_c$; since the latter is connected we must have the inclusion
\[ L\left(\mathcal{L}\SS^3(-\mathbf{1},\k)_c\right) \subset \mathcal{L}\SS^2(-\mathbf{1})_c. \]
This proves one part of the statement, namely that if $\gamma \in \mathcal{L}\SS^3(-\mathbf{1},\k)_c$, then its left part $\gamma_l=L(\gamma) \in \mathcal{L}\SS^2(-\mathbf{1})_c$. To prove the other part, it is enough to verify that    
\[ L\left(\mathcal{L}\SS^3(-\mathbf{1},\k)_n\right) \subset \mathcal{L}\SS^2(-\mathbf{1})_n. \]
To show this inclusion, using continuity and connectedness arguments as before, it is enough to find one element in $\mathcal{L}\SS^3(-\mathbf{1},\k)_n$ whose image by $L$ belongs to $\mathcal{L}\SS^2(-\mathbf{1})_n$. We claim that the curve $\gamma_1^5$ from Example~\ref{family1} does the job. To see that $\gamma_1^5 \in \mathcal{L}\SS^3(-\mathbf{1},\k)_n$, one can easily check that if we define the plane 
\[ H=\{(x_1,0,0,x_4) \in \R^4 \; | \; x_1 \in \R, \; x_4 \in \R\} \]  
then
\[ \gamma_1^5(t_i) \in H, \quad t_i=\frac{i}{5}, \quad 1 \leq i \leq 4. \]
Hence $\gamma_1^5$ has at least $4$ points of intersection with $H$; this shows that $\gamma_5$ is not convex. To conclude, it is clear that $L(\gamma_1^5)=\sigma_5 \in \mathcal{L}\SS^2(-\mathbf{1})_n$. Hence this proves the desired inclusion and concludes the proof.

\end{proof}  



\subsection{Proof of Theorem ~\ref{th3}}\label{s64}

Now we detail the proof of Theorem~\ref{th3}, which gives a necessary condition for a curve in $\mathcal{L}\SS^3(\1,-\1)$ to be locally convex by merely considering its left part. 
First we need to recall some basic definition and properties. 

An \emph{open hemisphere} $H$ in $\SS^2$ is a subset of $\SS^2$ of the form
$ H_h=\{x \in \SS^2 \; | \; h\cdot x >0\} $
for some $h \in \SS^2$, and a \emph{closed hemisphere} is the closure $\bar{H}$ of an open hemisphere, that is it has the form 
$ \bar{H}_h=\{x \in \SS^2 \; | \; h\cdot x \geq 0\}. $
We can make the following definition.
A closed curve $\gamma : [0,1] \rightarrow \SS^2$ is \emph{hemispherical} if it its image is contained in an open hemisphere of $\SS^2$. It is \emph{borderline hemispherical} if it is contained in a closed hemisphere but not contained in any open hemisphere.

Following~\cite{Zul12}, we define a rotation number for any closed curve $\gamma$ in $\SS^2$ contained in a closed hemisphere (such a curve is either hemispherical or borderline hemispherical). To such a closed curve $\gamma$ contained in a closed hemisphere, there is a distinguished choice of hemisphere $h_\gamma$ containing the image of $\gamma$ (this hemisphere $h_\gamma$ is the barycenter of the set of all closed hemisphere containing the image of $\gamma$, the latter being geodesically convex, see~\cite{Zul12} for further details). Let $\Pi_{h_\gamma} : \SS^2 \rightarrow \R^2$ be the stereographic projection from $-h_\gamma$, and $\eta_\gamma= \Pi_{h_\gamma} \circ \gamma$. The curve $\eta_\gamma$ is now a closed curve in the plane $\R^2$, and it is an immersion. The definition of its rotation number $\mathrm{rot}(\eta_\gamma) \in \Z$ is now classical: for instance, it can be defined to be the degree of the map
\[ t \in \SS^1 \mapsto \frac{\eta_\gamma'(t)}{||\eta_\gamma'(t)||} \in \SS^1. \]    
Given a closed curve contained in a closed hemisphere in $\SS^2$, its rotation number $\mathrm{rot}(\gamma)$ is defined by
$ \mathrm{rot}(\gamma):=-\mathrm{rot}(\eta_\gamma) \in \Z. $

The proof of Theorem~\ref{th3} will be based on two lemmas. The first lemma is a well-known property. 
For ease of presentation we omit the proof.

\begin{lemma}\label{lem1}
Consider a continuous map $H : [0,1] \rightarrow \mathcal{L}\SS^2(\1)$ such that $\gamma_0=H(0)$ has the property of being hemispherical with rotation number equal to $2$ and $\gamma_1=H(1)$ which does not have this property. Then there exists a time $t>0$ such that $\gamma_t=H(t)$ is borderline hemispherical with rotation number equal to $2$. 
\end{lemma} 

The next lemma will be proven below.

\begin{lemma}\label{lem2}
Consider the map $L : \mathcal{L}\SS^3(\1,-\1) \rightarrow \mathcal{L}\SS^2(\1)$ given by $L(\gamma)=\gamma_l$, and let $\mathcal{L}\SS^3(\1,-\1)_c$ be the set of convex curves. Then the image of $\mathcal{L}\SS^3(\1,-\1)_c$ by $L$ does not contain a borderline hemispherical curve with rotation number equal to $2$. 
\end{lemma} 



Now those lemmas build upon the setup of the problem to yield the proof of Theorem~\ref{th3}. 

\begin{proof}[Proof of Theorem~\ref{th3}]
Recall that the map $L : \mathcal{L}\SS^3(\1,-\1) \rightarrow \mathcal{L}\SS^2(\1)$ given by $L(\gamma)=\gamma_l$ is continuous, and that $\mathcal{L}\SS^3(\1,-\1)$ contains exactly two connected components, one of which is made of convex curves $\mathcal{L}\SS^3(\1,-\1)_c$ and the other of non-convex curves $\mathcal{L}\SS^3(\1,-\1)_n$.  We need to prove that the image of $\mathcal{L}\SS^3(\1,-\1)_c$ by $L$ contains only curves which are hemispherical with rotation number equal to $2$.

First let us prove that this image contains at least one such element. Recall the family of curves $\gamma_{1}^m \in \mathcal{L}\SS^3((-\1)^m,\k^m)$, $m \geq 1$, defined in Example~\ref{family1}. For $m=2$, the curve $\gamma_1^2=(\sigma_\pi^2,\sigma^1_{2\pi}) \in \mathcal{L}\SS^3(\1,-\1)$ is convex (see Section~\ref{examples}). 
Moreover, it is clear that $L(\gamma_1^2)=\sigma_\pi^2$ is hemispherical and has rotation number equal to $2$, and therefore the image of  $\mathcal{L}\SS^3(\1,-\1)_c$ by $L$ contains at least the curve $L(\gamma_1^2)=\sigma_\pi^2$. 

To prove that the image of $\mathcal{L}\SS^3(\1,-\1)_c$ by $L$ contains only curves which are hemispherical with rotation number equal to $2$, we argue by contradiction, and assume that the image of $\mathcal{L}\SS^3(\1,-\1)_c$ by $L$ contains a curve which is not hemispherical with rotation number equal to $2$. Since $L$ is continuous and $\mathcal{L}\SS^3(\1,-\1)_c$ is connected, its image by $L$ is connected and thus we can find a homotopy $H : [0,1] \rightarrow L\left(\mathcal{L}\SS^3(\1,-\1)_c\right) \subset \mathcal{L}\SS^2(\1)$ between $H(0)=\sigma_\pi^2$, which is hemispherical with rotation number equal to $2$, and a curve $H(1)$ which does not have this property. Using Lemma~\ref{lem1}, one can find a time $t>0$ such that $H(t) \in L\left(\mathcal{L}\SS^3(\1,-\1)_c\right)$ is borderline hemispherical with rotation number equal to $2$. But by Lemma~\ref{lem2}, such a curve $H(t)$ cannot belong to $ L\left(\mathcal{L}\SS^3(\1,-\1)_c\right)$, and so we arrive at a contradiction.   
\end{proof}

To conclude, it remains to prove Lemma~\ref{lem2}.

\begin{proof}[Proof of Lemma~\ref{lem2}]
We argue by contradiction, and assume that there exists a curve $\beta \in \mathcal{L}\SS^3(\1,-\1)_c$ (that is a convex curve $\beta \in \mathcal{L}\SS^3(\1,-\1)$) such that its left part $\beta_l$ is  borderline hemispherical curve with rotation number equal to $2$. 

First we use our assumption that $\beta$ is convex, which implies that $\mathcal{F}_{\beta}(t)$ belongs to the Bruhat cell of $A ^\top$ for all time $t \in [0,1]$ (see Proposition $64$ in~\cite{Alv16} or Theorem $3$ in~\cite{GouSal19I}), where
\[ A ^\top=
\begin{pmatrix}
0 & 0 & 0 & -1 \\
0 & 0 & 1 & 0 \\
0 & -1 & 0 & 0 \\
1 & 0 & 0 & 0 
\end{pmatrix}.
\]  
Therefore, by definition, there exist matrices $U_1(t) \in \mathrm{Up}_4^+$, $U_2(t) \in \mathrm{Up}_4^+$ (recall that $\mathrm{Up}_4^+$ is the group of upper triangular $4 \times 4$ matrices with positive diagonal entries) such that
\[ \mathcal{F}_{\beta}(t)=U_1(t){} A ^\top U_2(t).\]
This condition can also be written as
\[ \mathcal{F}_{\beta}(t)={}  A ^\top L_1(t)U_2(t), \quad L_1(t):=AU_1(t){}A ^\top,\]
with $L_1(t) \in \mathrm{Lo}_4^+$, where $\mathrm{Lo}_4^+$ is the group of lower triangular $4 \times 4$ matrices with positive diagonal entries. Such a decomposition is not unique, but there exists a unique decomposition
\begin{equation}\label{dec}
\mathcal{F}_{\beta}(t)={}  A ^\top L(t)U(t),
\end{equation}
where $U(t) \in \mathrm{Up}_4^+$. Now $L(t)\in \mathrm{Lo}_4^1$, where $\mathrm{Lo}_4^1$ is the group of lower triangular $4 \times 4$ matrices with diagonal entries equal to one. Using the fact that $\mathcal{F}_{\beta}(t)^{-1}\mathcal{F}_{\beta}'(t)$ belongs to $\mathfrak{J}$ (because $\beta$ is in particular locally convex), it is easy to see, by a simple computation, that the matrix $L(t)$ in~\eqref{dec} is such that $L(t)^{-1}L'(t)$ has positive subdiagonal entries and all other entries are zero. That is, we can write
\begin{equation}\label{L}
L(t)^{-1}L'(t)=
\begin{pmatrix}
    0 & 0 & 0 &  0 \\
    + & 0 & 0 &  0 \\
    0 & + & 0 &  0 \\
    0 & 0 & + &  0
\end{pmatrix}
, \quad t \in [0,1].
\end{equation}

Then we use our assumption that the left part $\beta_l$ is a borderline hemispherical curve with rotation number equal to $2$. 
\begin{figure}[H]
\centering
\includegraphics[scale=0.3]{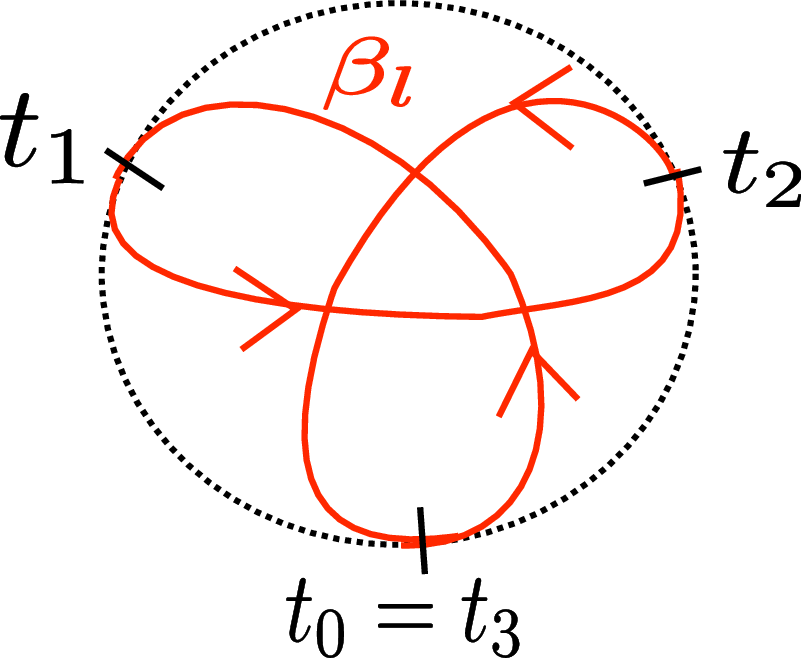}
\caption{The curve $\beta_l$.}
\label{fig:l}
\end{figure}
This implies (see Figure~\ref{fig:l}, where the dotted circle represents the equator of the sphere) that there exist times $t_1$ and $t_2$ and reals $\theta_1$ and $\theta_2$ such that
\[ \tilde{\mathcal{F}}_{\beta_l}(t_1)=\exp(\theta_1\k) \in \SS^3, \quad \tilde{\mathcal{F}}_{\beta_l}(t_2)=\exp(\theta_2\k) \in \SS^3.\]
Consequently, for $\beta$, we have
\begin{equation}\label{assum}
\begin{cases}
\tilde{\mathcal{F}}_{\beta}(t_1)=(\exp(\theta_1\k),z_r(t_1)) \in \SS^3 \times \SS^3 \\
\tilde{\mathcal{F}}_{\beta}(t_2)=(\exp(\theta_2\k),z_r(t_2)) \in \SS^3 \times \SS^3.
\end{cases}
\end{equation}

Following Subsection $4.1$ in~\cite{AlvSal19}, let us denote by $\k_l$ the matrix in $\mathfrak{so}_4$ that corresponds to the left multiplication by $\k \in \mathbb{H}$. This matrix is given by 
\[ \k_l=
\begin{pmatrix}
0 & 0 & 0 & -1 \\
0 & 0 & -1 & 0 \\
0 & +1 & 0  & 0 \\
+1 & 0 & 0 & 0
\end{pmatrix}. \]
Recalling that $\Pi_4 : \SS^3 \times \SS^3 \rightarrow \mathrm{SO}_4$ is the canonical projection, it follows from~\eqref{assum} that $\mathcal{F}_{\beta}(t_1)=\Pi_4(\tilde{\mathcal{F}}_{\beta}(t_1))$ and $\mathcal{F}_{\beta}(t_2)=\Pi_4(\tilde{\mathcal{F}}_{\beta}(t_2))$ belong to the subgroup $H$ of matrices in $\mathrm{SO}_4$ that commutes with the matrix $\k_l$. Clearly, this subgroup $H$ consists of matrices of the form
\[
\begin{pmatrix}
q_{11} & q_{12} & -q_{42} & -q_{41} \\
q_{21} & q_{22} & -q_{32} & -q_{31} \\
q_{31} & q_{32} & q_{22}  & q_{21} \\
q_{41} & q_{42} & q_{12} & q_{11}
\end{pmatrix} \in \mathrm{SO}_4. \]
Using this explicit form of $H$ and the fact that $\mathcal{F}_{\beta}(t_1) \in H$ and $\mathcal{F}_{\beta}(t_2) \in H$, one finds, after a direct computation, that the matrix 
\begin{equation}
L(t)=
\begin{pmatrix}
1 & 0 & 0 & 0 \\
l_{21}(t) & 1 & 0 & 0 \\
l_{31}(t) & l_{32}(t) & 1  & 0 \\
l_{41}(t) & l_{42}(t) & l_{43}(t) & 1
\end{pmatrix}
\end{equation}
defined in~\eqref{dec} satisfies, at $t=t_1$ and $t=t_2$, the conditions
\begin{equation}\label{contra}
l_{21}(t_1)=-l_{43}(t_1), \quad l_{21}(t_2)=-l_{43}(t_2). 
\end{equation}
But clearly, \eqref{contra} is not compatible with~\eqref{L}, and this gives the desired contradiction.  

\end{proof}

\section{Final Considerations}

The study of the spaces of locally convex curves started in the seventies with the works of Litte on the $2$-sphere. 
But the research on the topological aspects on these spaces of curves on the spheres of higher dimension as in related spaces is very productive area, here we mention some other relevant works: \cite{Dub61}, \cite{MS012}, \cite{SZ13}, \cite{SZ15}, \cite{SZ16}, \cite{Sma58}, \cite{Sma59b}, \cite{Sma59}, \cite{Whi37} and \cite{Zhou}.
A very hard and interesting question in this topic is to determine the homotopy type of the spaces of locally convex curves on the $n$-sphere, for $n \geq 3$.
In this section we will give some directions of future research and some conjectures.


In~\cite{SS12}, N. Saldanha and B. Shapiro proved that the spaces $\mathcal{L}\SS^n(Q)$ fall in at most $\lceil \frac{n}{2} \rceil + 1$ equivalence classes up to homeomorphism, they also studied this classification in the double cover $\mathrm{Spin}_{n+1}$. Therefore, one natural question is to determine if the listed spaces are pairwise non-homemorphic.

The list in the case $n=2$ says that $\mathcal{L}\SS^2(Q)$ is homeomorphic to one of these 2 spaces
$\Omega(\mathrm{SO}_{3}), \;  \mathcal{L}\SS^2(I)$; and 
$\mathcal{L}\SS^2(z)$ is homeomorphic to one of these three spaces
$\Omega \SS^3, \; \mathcal{L}\SS^2(\1), \; \mathcal{L}\SS^2(-\1).$
In this case, all the listed spaces are non-homeomorphic.
Moreover, in~\cite{Sal13} the following homotopy equivalences are proven to hold
\[ \mathcal{L}\SS^2(\mathbf{1}) \approx (\Omega \SS^3) \vee \SS^2 \vee \SS^6 \vee \SS^{10} \vee \cdots, \quad \mathcal{L}\SS^2(-\mathbf{1})_n \approx (\Omega \SS^3) \vee \SS^4 \vee \SS^8 \vee \cdots.\]


In the case $n=3$, i.e., in the case of $\mathrm{SO}_{4}$ there are at most 3 equivalence classes, and in the case of $\SS^3 \times \SS^3$ at most 5.
Therefore, $\mathcal{L}\SS^3(Q)$ is homeomorphic to one of these three spaces
$\mathcal{L}\SS^3(-I), \; \Omega(\mathrm{SO}_4), \; \mathcal{L}\SS^3(I)$;
and $\mathcal{L}\SS^3(z)$ is homeomorphic to one of these five spaces
$\mathcal{L}\SS^3(-\1,\1), \; \mathcal{L}\SS^3(\1,-\1), \; \Omega(\SS^3 \times \SS^3), \;
 \mathcal{L}\SS^3(\1,\1), \; \mathcal{L}\SS^3(-\1,-\1) . $
In particular, we have $\mathcal{L}\SS^3(-\1,\k) \simeq \mathcal{L}\SS^3(\1,-\1)$ and therefore we believe that a stronger version of Lemma~\ref{lem2} is true:
\begin{Conj} 
The image of the whole space $\mathcal{L}\SS^3(\1,-\1)$ by $L$ does not contain a borderline hemispherical curve with rotation number equal to $2$. 
\end{Conj}
With this stronger statement it would be easy to see from the proof of Theorem~\ref{th3} that our necessary condition for a curve in $\mathcal{L}\SS^3(\1,-\1)$ to be convex is also sufficient. 
Yet for the moment we are not able to prove this stronger statement.

Furthermore, using some techniques developed in~\cite{GouSal19I},~\cite{GouSal19II} and~\cite{AlvSal19} we hope to proof the conjecture below, in particular solving the main problem in the case $n=3$, this is a joint work with V. Goulart, N. Saldanha and B. Shapiro. 

\begin{Conj}

We have the following weak homotopy equivalences:
\begin{align*}
\mathcal{L}\SS^3(+\1,+\1) &\approx 
\Omega(\SS^3 \times \SS^3) \vee \SS^4 \vee \SS^8 \vee \SS^8
\vee \SS^{12} \vee \SS^{12} \vee \SS^{12} \vee \cdots, \\
\mathcal{L}\SS^3(-\1,-\1) &\approx 
\Omega(\SS^3 \times \SS^3) \vee \SS^2 \vee \SS^6 \vee \SS^6
\vee \SS^{10} \vee \SS^{10} \vee \SS^{10} \vee \cdots, \\
\mathcal{L}\SS^3(+\1,-\1) &\approx 
\Omega(\SS^3 \times \SS^3) \vee \SS^0 \vee \SS^4 \vee \SS^4
\vee \SS^{8} \vee \SS^{8} \vee \SS^{8} \vee \cdots, \\
\mathcal{L}\SS^3(-\1,+\1) &\approx 
\Omega(\SS^3 \times \SS^3) \vee \SS^2 \vee \SS^6 \vee \SS^6
\vee \SS^{10} \vee \SS^{10} \vee \SS^{10} \vee \cdots.
\end{align*}
The above bouquets include one copy of $\SS^k$,
two copies of $\SS^{(k+4)}$, \dots, $j+1$ copies of $\SS^{(k+4j)}$, \dots,
and so on.

\end{Conj}


\bigskip

\parindent=0pt
\parskip=0pt
\obeylines

Em\'ilia Alves 
emiliacstalves@gmail.com 
Instituto de Matem\'atica e Estat\'istica
Universidade Federal Fluminense
Rua Professor Marcos Waldemar de Freitas Reis s/n
24210-201 Niter\'oi, RJ, Brazil

 


\end{document}